\documentclass[letterpaper, 10 pt, conference]{ieeeconf}  

\IEEEoverridecommandlockouts                              
\IEEEoverridecommandlockouts                              
\usepackage{times} 
\usepackage{amsmath} 
\usepackage{amssymb}  

\usepackage{graphicx}  
\usepackage{amsmath}   
\usepackage{array}
\usepackage{times}
\usepackage{comment}
\usepackage{latexsym,theorem}
\usepackage{amsmath,amsfonts,amssymb,amsbsy}
\usepackage[ruled]{algorithm2e}

\newcommand{\tr}{\top}
\newcommand{\Nset}{\mathbb{N}}
\newcommand{\Rset}{\mathbb{R}}
\newcommand{\Uset}{\mathbb{U}}

\newcommand{\Yset}{\mathbb{Y}}

\newcommand{\cN}{\mathcal{N}}
\newcommand{\cK}{\mathcal{K}}

\newcommand{\tini}{\text{ini}}
\newcommand{\bu}{\mathbf{u}}
\newcommand{\br}{\mathbf{r}}
\newcommand{\bx}{\mathbf{x}}
\newcommand{\by}{\mathbf{y}}
\newcommand{\bg}{\mathbf{g}}
\newcommand{\bY}{\mathbf{Y}}
\newcommand{\bU}{\mathbf{U}}
\newcommand{\spc}{\text{spc}}

\newcommand{\col}{\operatorname{col}}%

\newcommand{\diag}{\operatorname{diag}}

{\theorembodyfont{\slshape}\newtheorem{theorem}{Theorem}[section]}
{\theorembodyfont{\slshape}}
{\theorembodyfont{\slshape}\newtheorem{lemma}[theorem]{Lemma}}
{\theorembodyfont{\slshape}}
{\theorembodyfont{\slshape}}
{\theorembodyfont{\upshape}\newtheorem{definition}[theorem]{Definition}}
{\theorembodyfont{\upshape}}
{\theorembodyfont{\upshape}\newtheorem{remark}[theorem]{Remark}}
{\theorembodyfont{\upshape}\newtheorem{assumption}[theorem]{Assumption}}

\title{\LARGE \bf
	Generalized Data--driven Predictive Control
}

\author{M. Lazar and P. C. N. Verheijen
	\thanks{The authors are with the Department of Electrical Engineering, Eindhoven University of Technology, The Netherlands, E-mails:
		{\tt\small m.lazar@tue.nl, p.c.n.verheijen@tue.nl}.}%
}

\begin{document}
	
\maketitle
\thispagestyle{empty}
\pagestyle{empty}

\begin{abstract}Data--driven predictive control (DPC) is becoming an attractive alternative to model predictive control as it requires less system knowledge for implementation and reliable data is increasingly available in smart engineering systems. Two main approaches exist within DPC, which mostly differ in the construction of the predictor: estimated prediction matrices (unbiased for large data) or Hankel data matrices as predictor (allows for optimizing the bias/variance trade--off). In this paper we develop a novel, generalized DPC (GDPC) algorithm that constructs the predicted input sequence as the sum of a known input sequence and an optimized input sequence. The predicted output corresponding to the known input sequence is computed using an unbiased, least squares predictor, while the optimized predicted output is computed using a Hankel matrix based predictor. By combining these two types of predictors, GDPC can achieve high performance for noisy data even when using a small Hankel matrix, which is computationally more efficient. Simulation results for a benchmark example from the literature show that GDPC with a minimal size Hankel matrix can match the performance of data--enabled predictive control with a larger Hankel matrix in the presence of noisy data.
\end{abstract}

\section{Introduction}
\label{sec1}
Reliable data is becoming increasingly available in modern, smart engineering systems, including mechatronics, robotics, power electronics, automotive systems, and smart infrastructures, see, e.g., \cite{Hou_2013, Lagarrigue_2017} and the references therein. For these application domains, model predictive control (MPC) \cite{Maciejowski_2002, Camacho_2007} has become the preferred advanced control method for several reasons, including constraints handling, anticipating control actions, and optimal performance. Since obtaining and maintaining accurate models requires effort and reliable data becomes readily available in engineering systems, it is of interest to develop data--driven predictive control (DPC) algorithms that can be implemented in practice.  An indirect data--driven approach to predictive control design was already developed in \cite{FavoreelSPC1999} more than 20 years ago, i.e.,  subspace predictive control (SPC). The SPC approach skips the identification of the prediction model and identifies the complete prediction matrices from input--output data using least squares. This provides an unbiased predictor for sufficiently large data.

More recently, a direct data--driven approach to predictive control design was developed in \cite{CoulsonDeePC2019} based on Willems' fundamental lemma \cite{Willems_2005}, i.e., data--enabled predictive control (DeePC). The idea to use (reduced order) Hankel matrices as predictors has been put forward earlier in \cite{Yang_DeePC_2015}, but the first well--posed constrained data--enabled predictive control algorithm was formulated in \cite{CoulsonDeePC2019}, to the best of the authors' knowledge. The DeePC approach skips the identification of prediction models or matrices all together and utilizes Hankel matrices built from input--output data to parameterize predicted future inputs and outputs. In the deterministic, noise free case, equivalence of MPC and DeePC was established in \cite{CoulsonDeePC2019, Berberich_2020}, while equivalence of SPC and DeePC was shown in \cite{Fiedler_2021}. Stability guarantees for DeePC were first obtained in \cite{Berberich_2020} by means of terminal equality constraints and input--output--to--state stability Lyapunov functions. Alternatively, stability guarantees for DeePC were provided in \cite{LazarNMPC2021} using terminal inequality constraints and dissipation inequalities involving storage and supply functions.  An important contribution to DeePC is the consistent regularization cost introduced in \cite{Dorfler_PI}, which enables reliable performance in the presence of noisy data. Indeed, since the DeePC algorithm jointly solves estimation and controller synthesis problems, the regularization derived in \cite{Dorfler_PI} allows one to optimize the bias/variance trade--off if data is corrupted by noise. A systematic method for tuning the regularization cost weighting parameter was recently presented in \cite{LazarCDC_2022}.

Computationally, SPC has the same number of optimization variables as MPC, which is equal to the number of control inputs times the prediction horizon. In DeePC, the number of optimization variables is driven by the data length, which is in general much larger than the prediction horizon. Especially in the case of noisy data, a large data size is required to attain reliable predictions, see, e.g., \cite{Berberich_2020, Dorfler_PI, LazarCDC_2022}. As this hampers real--time implementation, it is of interest to improve computational efficiency of DeePC. In \cite{Breschi_2023_Auto}, a computationally efficient formulation of DeePC was provided via LQ factorization of the Hankel data matrix, which yields the same online computational complexity as SPC/MPC. In this approach, DeePC yields an unbiased predictor, similar to SPC. In \cite{Kaixiang_2022}, a singular value decomposition is performed on the original Hankel data matrix and a DeePC algorithm is designed based on the resulting reduced Hankel matrix. Therein, it was shown that this approach can significantly reduce the computational complexity of DeePC, while improving the accuracy of predictions for noisy data. In \cite{Baros_2022}, an efficient numerical method that exploits the structure of Hankel matrices was developed for solving quadratic programs (QPs) specific to DeePC. Regarding real--life applications of DeePC, the minimal data size required for persistence of excitation is typically used, see, e.g., \cite{Elokda_2021, Berberich_2021_exp}, or an unconstrained solution of DeePC is used instead of solving a QP, see, e.g., \cite{Carlet_2022}. These approaches however limit the achievable performance in the presence of noisy data and hard constraints, respectively.

In this paper we develop a novel, generalized DPC (GDPC) algorithm that constructs the predicted input sequence as the sum of a known input sequence and an optimized input sequence. The predicted output corresponding to the known input sequence is computed using an unbiased, least squares predictor based on a large data set. The optimized predicted output is computed using a Hankel matrix based on a smaller (possible different) data set. Based on the extension of Willems' fundamental lemma to multiple data sets \cite{vanWaarde_2020}, the sum of the two trajectories spanned by two (possibly different) data sets will remain a valid system trajectory, as long as the combined data matrices are collectively persistently exciting and the system is linear. By combining these two types of predictors, GDPC can achieve high performance in the presence of noisy data even when using Hankel matrices of smaller size, which is computationally efficient. The performance and computational complexity of GDPC with a minimal (according to DeePC design criteria) size Hankel matrix is evaluated for a benchmark example from the MPC literature and compared to DeePC with a Hankel matrix of varying size. 

The remainder of this paper is structured as follows. The necessary notation and the DeePC approach to data--driven predictive control are introduced in Section~\ref{sec2}. The GDPC algorithm is presented in Section~\ref{sec3}, along with design guidelines, stability analysis and other relevant remarks. Simulation results and a comparison with DeePC are provided in Section~\ref{sec4} for a benchmark example from the literature. Conclusions are summarized in Section~\ref{sec5}.

\section{Preliminaries}
\label{sec2}
Consider a discrete--time linear dynamical system subject to zero--mean Gaussian noise $w(k)\sim \cN(0,\sigma_w^2 I)$:
\begin{equation}
\label{eq:2.1}
\begin{split}
x(k+1)&=Ax(k)+Bu(k), \quad k\in\Nset,\\
y(k)&=C x(k)+w(k),\end{split}
\end{equation}
where $x\in\Rset^n$ is the state, $u\in\Rset^{n_u}$ is the control input, $y\in\Rset^{n_y}$ is the measured output and $(A,B,C)$ are real matrices of suitable dimensions. We assume that $(A,B)$ is controllable and $(A,C)$ is observable. By applying a persistently exciting input sequence $\{u(k)\}_{k\in\Nset_{[0,T]}}$ of length $T$ to system \eqref{eq:2.1} we obtain a corresponding output sequence $\{y(k)\}_{k\in\Nset_{[0,T]}}$. 

If one considers an input--output model corresponding to \eqref{eq:2.1}, it is necessary to introduce the parameter $T_\tini$ that limits the window of past input--output data necessary to compute the current output, i.e.,
\begin{equation}
	\label{eq:2:arx}
	y(k)=\sum_{i=1}^{T_\tini} a_iy(k-i)+\sum_{i=1}^{T_\tini }b_iu(k-i),
	\end{equation}
for some real--valued coefficients. For simplicity of exposition we assume the same $T_\tini$ for inputs and outputs.

Next, we introduce some instrumental notation. For any finite number $q\in\Nset_{\geq 1}$ of vectors $\{\xi_1,\ldots,\xi_q\}\in\Rset^{n_1}\times\ldots\times\Rset^{n_q}$ we will make use of the operator $\col(\xi_1,\ldots,\xi_q):=[\xi_1^\tr,\ldots,\xi_q^\tr]^\tr$. For any $k\geq 0$ (starting time instant in the data vector) and  $j\geq 1$ (length of the data vector), define
\begin{align*}
	\bar\bu(k,j) &:= \col(u(k),\ldots, u(k+j-1)),\\
    \bar\by(k,j)&: = \col(y(k), \ldots, y(k+j-1)).
\end{align*}
Let $N\geq T_\tini$ denote the prediction horizon. Then we can define the Hankel data matrices:
\begin{equation}\label{eq:hankel_data}
	\begin{aligned}
		\bU_p &:= \begin{bmatrix}\bar\bu(0,T_\tini) & \ldots & \bar\bu(T-1, T_\tini) \end{bmatrix}, \\
		\bY_p &:= \begin{bmatrix}\bar\by(1, T_\tini) & \ldots & \bar\by(T, T_\tini) \end{bmatrix},\\
		\bU_f &:= \begin{bmatrix}\bar\bu(T_\tini, N) & \ldots & \bar\bu(T_\tini+T-1, N) \end{bmatrix},\\
		\bY_f &:= \begin{bmatrix}\bar\by(T_\tini+1, N) & \ldots & \bar\by(T_\tini+T, N) \end{bmatrix}. 
	\end{aligned}
\end{equation}

According to the DeePC design \cite{CoulsonDeePC2019}, one must choose $T_\tini\geq n$ and $T\geq (n_u+1)(T_\tini+N+n)-1$, which implicitly requires an assumption on the system order (number of states). Given the measured output $y(k)$ at time $k\in\Nset$ and $T_\tini\in\Nset_{\geq 1}$ we define the sequences of known input--output data at time $k\geq T_\tini$, which are trajectories of system \eqref{eq:2.1}:  
\[\begin{split}
	\bu_\tini(k)&:=\col(u(k-T_\tini),\ldots,u(k-1)),\\
	\by_\tini(k)&:=\col(y(k-T_\tini+1),\ldots,y(k)).
\end{split}\]
Next, we define the sequences of predicted inputs and outputs at time $k\geq T_\tini$, which should also be trajectories of system \eqref{eq:2.1}:
\[\begin{split}
	\bu(k)&:=\col(u(0|k),\ldots,u(N-1|k)),\\
	\by(k)&:=\col(y(1|k),\ldots,y(N|k)). 
\end{split}\]

For a positive definite matrix $L$ let $L^\frac{1}{2}$ denote its Cholesky factorization. At time $k\geq T_\tini$, given $\bu_\tini(k), \by_\tini(k)$, the regularized DeePC algorithm \cite{Dorfler_PI} computes a sequence of predicted inputs and outputs as follows:
\begin{subequations}
	\label{eq:2.DeePC}
	\begin{align}
		\min_{\bg(k),\bu(k),\by(k),\sigma(k)}  l_N(y(N|k))&+\sum_{i=0}^{N-1}l(y(i|k),u(i|k))\nonumber\\&+\lambda_g l_g(\bg(k))+\lambda_\sigma l_\sigma(\sigma(k))\label{eq:2.DeePCa}\\
		&\text{subject to constraints:}\nonumber\\
		\begin{bmatrix}\bU_p \\\bY_p \\ \bU_f \\ \bY_f\end{bmatrix} \bg(k) &= \begin{bmatrix}\bu_\tini(k) \\ \by_\tini(k)+\sigma(k)\\ \bu(k)\\\by(k)\end{bmatrix},\label{eq:2.DeePCb}\\
		(\by(k),\bu(k))&\in\Yset^N\times\Uset^N.\label{eq:3.DeePCc}
	\end{align}
\end{subequations}
Above 
\begin{equation}
\label{eq:2:cost}
l(y,u):=\|Q^{\frac{1}{2}}(y-r_y)\|_2^2+\|R^{\frac{1}{2}}(u-r_u)\|_2^2,\quad l_\sigma(\sigma):=\|\sigma\|_2^2
\end{equation}
for some positive definite $Q,R$ matrices. The terminal cost is typically chosen larger than the output stage cost, to enforce convergence to the reference; a common choice is a scaled version of the output stage cost, i.e., $l_N(y):=\alpha l (y,0)$, $\alpha\geq 1 $. The references $r_y\in\Rset^{n_y}$ and $r_u\in\Rset^{n_u}$ can be constant or time--varying. We assume that the sets $\Yset$ and $\Uset$ contain $r_y$ and $r_u$ in their interior, respectively. The cost
\begin{equation}
\label{eq:2.lg}
l_g(\bg):=\|(I-\Pi)\bg\|_2^2,\quad \Pi:=\begin{bmatrix}\bU_p \\\bY_p \\ \bU_f\end{bmatrix}^\dagger \begin{bmatrix}\bU_p \\\bY_p \\ \bU_f\end{bmatrix}, 
\end{equation}
is a regularization cost proposed in \cite{Dorfler_PI}, where $[\cdot]^\dagger$ denotes a generalized pseudo--inverse. Notice that using such a regularization cost requires $T\geq T_\tini(n_u+n_y)+Nn_u$ in order to ensure that the matrix $I-\Pi$ has a sufficiently large null--space. If a shorter data length $T$ is desired, alternatively, the regularization cost $l_g(\bg):=\|\bg\|_2^2$ can be used. However, this regularization is not consistent, as shown in \cite{Dorfler_PI}. 

In the deterministic, noise--free case, the DeePC algorithm \cite{CoulsonDeePC2019} does not require the costs $l_g, l_\sigma$ and the variables $\sigma$. We observe that the computational complexity of DeePC is dominated by the vector of variables $\bg\in\Rset^T$, with $T\geq (n_u+1)(T_\tini+N+n)-1$. Hence, ideally one would prefer to work with the minimal value of data length $T$, but in the presence of noise, typically, a rather large data length $T$ is required for accurate predictions \cite{Berberich_2020, Dorfler_PI, LazarCDC_2022}.

\section{Generalized data--driven predictive control}
\label{sec3}
In this section we develop a novel, generalized DPC algorithm by constructing the predicted input sequence $\bu(k)$ as the sum of two input sequences, i.e., 
\begin{equation}
\bu(k):=\bar\bu(k)+\bu_g(k),\quad k\in\Nset,
\end{equation}
where $\bar\bu$ is a known, base line input sequence typically chosen as the shifted, optimal input sequence from the previous time, i.e.,
\begin{equation}
\label{eq:3:ufree}
\bar\bu(k):=\{u^\ast(1|k-1),\ldots,u^\ast(N-1|k-1),\bar u(N-1|k)\},
\end{equation}
where common choices for the last element $\bar u(N-1|k)$ are $r_u$ or $u^\ast(N-1|k-1)$. At time $k=T_\tini$, when enough input--output data is available to run the GDPC algorithm, $\bar\bu(k)$ is initialized using a zero input sequence (or an educated guess). The sequence of inputs $\bu_g$ can be freely optimized online by solving a QP, as explained next. 

Using a single persistently exciting input sequence split into two parts, or two different persistently exciting input sequences, we can define two Hankel data matrices as in \eqref{eq:hankel_data}, i.e.,
\begin{equation}
\label{eq:3:Hankel}
\begin{split}
\bar H&:=\begin{bmatrix}\bar\bU_{p} \\\bar\bY_{p} \\ \bar\bU_{f} \\ \bar\bY_{f}\end{bmatrix}\in\Rset^{(n_u+n_y)(T_\tini+N)\times T}, \\ H&:=\begin{bmatrix}\bU_p \\\bY_p \\ \bU_f \\ \bY_f\end{bmatrix}\in\Rset^{(n_u+n_y)(T_\tini+N)\times T_g},
\end{split}
\end{equation}
where the length $T$ of the first part/sequence can be taken as large as desired, an the choice of the length $T_g$ of the second part/sequence is flexible. I.e., $T_g$ should be small enough to meet computational requirements, but it should provide enough degrees of freedom to optimize the bias/variance trade off. Offline, compute the matrix
\begin{equation}
\Theta:=\bar\bY_f\begin{bmatrix}\bar\bU_{p} \\\bar\bY_{p} \\ \bar\bU_{f} \end{bmatrix}^\dagger\in\Rset^{(n_yN)\times(T_\tini(n_u+n_y)+n_uN)}.
\end{equation}
Online, at time $k\geq T_\tini$, given $\bu_\tini(k)$, $\by_\tini(k)$ and $\bar\bu(k)$, compute
\begin{equation}
\label{eq:3.gbase}
   \bar\by(k)=\Theta \begin{bmatrix} \bu_\tini(k)\\\by_\tini(k)\\ \bar\bu(k)\end{bmatrix}
\end{equation}
and solve the \emph{GDPC optimization problem}:
\begin{subequations}
	\label{eq:3.GDPC}
	\begin{align}
		\min_{\bg(k),\bu(k),\by(k),\sigma(k)}  l_N(y(N|k))&+\sum_{i=0}^{N-1}l(y(i|k),u(i|k))\nonumber\\&+\lambda_g l_g(\bg(k))+\lambda_\sigma l_\sigma(\sigma(k))\label{eq:3.GDPCa}\\
		\text{subject to constraints:}&\nonumber\\
		\begin{bmatrix}\bU_p \\\bY_p \\ \bU_f \\ \bY_f\end{bmatrix} \bg(k) &= \begin{bmatrix}0 \\ \sigma(k)\\ \bu(k)-\bar\bu(k)\\\by(k)-\bar\by(k)\end{bmatrix},\label{eq:3.GDPCb}\\
		(\by(k),\bu(k))&\in\Yset^N\times\Uset^N.\label{eq:3.GDPCc}
	\end{align}
\end{subequations}
Above, the cost functions $l_N(y)$, $l(y,u)$, $l_g(g)$ and $l_\sigma(\sigma)$ are defined in the same way as in \eqref{eq:2:cost} for some $Q,R\succ 0$. 

Since the size of the matrix $\Theta$ does not depend on the data length $T$, i.e., the number of columns of the Hankel matrix $\bar H$, the computation as in \eqref{eq:3.gbase} of the predicted output corresponding to $\bar\bu(k)$ is efficient even for a large $T$. Thus, GDPC benefits from an unbiased base line output prediction, which allows choosing $T_g$, i.e., the number of columns of the Hankel matrix $H$, much smaller than $T$. In turn, this reduces the online computational complexity of GDPC, without sacrificing performance in the presence of noisy data. It can be argued that the selection of $T_g$ provides a trade--off between computational complexity and available degrees of freedom to optimize the bias/variance trade off. 
\begin{remark}[\emph{Offset--free GDPC design}] In practice it is of interest to achieve offset--free tracking. Following the offset--free design for SPC developed in \cite{Verheijen_2021}, which was further applied to DeePC in \cite{LazarCDC_2022}, it is possible to design an offset--free GDPC algorithm by defining an incremental input sequence
\[
\Delta\bu(k):=\Delta\bar\bu(k)+\Delta \bu_g(k),\quad k\in\Nset,
\]
where $\Delta\bar\bu$ is chosen as the shifted optimal input sequence from the previous time, i.e.,
\begin{equation}
\label{eq:3:dufree}
\begin{split}
&\Delta \bar\bu(k):=\\&\{\Delta u^\ast(1|k-1),\ldots,\Delta u^\ast(N-1|k-1), \Delta \bar u(N-1|k)\}.
\end{split}
\end{equation}
The input data blocks in the Hankel matrices $\bar H$ and $H$ must be replaced with incremental input data, i.e., $\Delta \bar \bU_p$, $\Delta \bar \bU_f$ and $\Delta \bU_p$, $\Delta \bU_f$, respectively. The input applied to the system is then $u(k):=\Delta u^\ast(0|k)+u(k-1)$.
\end{remark}

In what follows we provide a formal analysis of the GDPC algorithm.
\subsection{Well-posedness and design of GDPC}
In this subsection we show that in the deterministic case GDPC predicted trajectories are trajectories of system \eqref{eq:2.1}. In this case, the GDPC optimization problem can be simplified as
\begin{subequations}
	\label{eq:3.sGDPC}
	\begin{align}
		\min_{\bg(k),\bu(k),\by(k)}  l_N(y(N|k))&+\sum_{i=0}^{N-1}l(y(i|k),u(i|k))\label{eq:3.sGDPCa}\\
		\text{subject to constraints:}&\nonumber\\
		\begin{bmatrix}\bU_p \\\bY_p \\ \bU_f \\ \bY_f\end{bmatrix} \bg(k) &= \begin{bmatrix}0 \\ 0\\ \bu(k)-\bar\bu(k)\\\by(k)-\bar\by(k)\end{bmatrix},\label{eq:3.sGDPCb}\\
		(\by(k),\bu(k))&\in\Yset^N\times\Uset^N.\label{eq:3.sGDPCc}
	\end{align}
\end{subequations}
\begin{lemma}[\emph{GDPC well--posedeness}]
\label{lem:DeeGPC}
Consider one (or two) persistently exciting input sequence(s) of sufficient length(s) and construct two Hankel matrices $\bar H$ and $H$ as in \eqref{eq:3:Hankel} with $T$ and $T_g$ columns, respectively, and such that $\bar H$ has full row rank. Consider also the corresponding output sequence(s) generated using system \eqref{eq:2.1}. For any given input sequence $\bar\bu(k)$, and initial conditions $\bu_\tini(k)$ and $\by_\tini(k)$, let $\bar\by(k)$ be defined as in \eqref{eq:3.gbase}. Then there exists a real vector $\bg(k)\in\Rset^{T_g}$ such that \eqref{eq:3.sGDPCb} holds if and only if $\bu(k)$ and $\by(k)$ are trajectories of system \eqref{eq:2.1}. 
\end{lemma}
\begin{proof}
Define $\bar\bg(k):=\begin{bmatrix}\bar\bU_{p} \\\bar\bY_{p} \\ \bar\bU_{f} \end{bmatrix}^\dagger\begin{bmatrix} \bu_\tini(k)\\\by_\tini(k)\\ \bar\bu(k)\end{bmatrix}$. Then it holds that:
\begin{align*}
\begin{bmatrix}\bar H & H\end{bmatrix}\begin{bmatrix}\bar \bg(k) \\ \bg(k)\end{bmatrix}&=\begin{bmatrix}\bar\bU_{p} \\\bar\bY_{p} \\ \bar\bU_{f} \\ \bar\bY_{f}\end{bmatrix}\bar\bg(k)+\begin{bmatrix}\bU_p \\\bY_p \\ \bU_f \\ \bY_f\end{bmatrix}\bg(k)\\&=\begin{bmatrix} \bu_\tini(k)\\\by_\tini(k)\\ \bar\bu(k)\\\bar\by(k)\end{bmatrix}+\begin{bmatrix} 0\\0\\ \bu(k)-\bar\bu(k)\\ \by(k)-\bar\by(k)\end{bmatrix}\\&=\begin{bmatrix} \bu_\tini(k)\\\by_\tini(k)\\ \bu(k)\\\by(k)\end{bmatrix}.
\end{align*}
Since the matrix $\bar H$ has full row rank, the concatenated matrix $\begin{bmatrix}\bar H & H\end{bmatrix}$ has full row rank and as such, the claim follows from \cite{CoulsonDeePC2019} if one input sequence is used and from  \cite{vanWaarde_2020} if two different input sequences are used to build the Hankel matrices.
\end{proof}

The selection of $T_g$ enables a trade off between computational complexity and available degrees of freedom to improve the output sequence generated by the known, base line input sequence. Indeed, a larger $T_g$ results in a larger null space of the data matrix $\begin{bmatrix}\bU_p\\\bY_p\end{bmatrix}$, which confines $\bg(k)$ in the deterministic case. However, high performance can be achieved in the case of noisy data even for a smaller $T_g$, because the base line predicted output is calculated using an unbiased least squares predictor, i.e., as defined in \eqref{eq:3.gbase}. 

An alternative way to define the known input sequence $\bar\bu(k)$ is to use an unconstrained SPC control law \cite{FavoreelSPC1999}. To this end, notice that the matrix $\Theta$ can be partitioned, see, e.g., \cite{Verheijen_2021}, into $\begin{bmatrix}P_1 & P_2 & \Gamma \end{bmatrix}$ such that
\[\bar\by(k)=\begin{bmatrix}P_1 & P_2\end{bmatrix}\begin{bmatrix}\bu_\tini(k)\\ \by_\tini(k)\end{bmatrix}+\Gamma \bar\bu(k).\]
Then, by defining $\Psi:=\diag\{R,\ldots,R\}$, $\Omega:=\diag\{Q,\ldots,Q,\alpha Q\}$, $G:=2\left(\Psi+\Gamma^T\Omega\Gamma\right)$ and $F:=2\Gamma^T\Omega$, we obtain:
\begin{equation}
\label{eq:3:freespc}
\begin{split}
&\bar\bu_{\spc}(k)\\&:=-G^{-1}\left(F\left(\begin{bmatrix}P_1&P_2\end{bmatrix}\begin{bmatrix}\bu_\tini(k)\\ \by_\tini(k)\end{bmatrix}-\br_y\right)-2\Psi\br_u\right),
\end{split}
\end{equation}
where $\br_y:=\col(r_y,\ldots,r_y)$ and $\br_u:=\col(r_u,\ldots,r_u)$. Since the inverse of $G$ is computed offline, computing $\bar\bu_\spc(k)$ online is numerically efficient even for a large data length $T$. In this case, since the corresponding  base line predicted output trajectory is unbiased, the simpler regularization cost $l_g(\bg):=\|\bg\|_2^2$ can be used in \eqref{eq:3.GDPCa}, without loosing consistency. When the base line input sequence is computed as in \eqref{eq:3:freespc}, the optimized input sequence $\bu_g(k)$ acts to enforce constraints, when the unconstrained SPC trajectories violates constraints, and it can also optimize the bias/variance trade off under appropriate tuning of $\lambda_g$. 

\subsection{Stability of GDPC} 
In this section we will provide sufficient conditions under which GDPC is asymptotically stabilizing. To this end define $J(\by(k),\bu(k)):=l_N(y(N|k))+\sum_{i=0}^{N-1}l(y(i|k),u(i|k))$, let $l_N(y):=l(y,0)$ and let $\by^\ast(k)$ and $\bu^\ast(k)$ denote optimal trajectories at time $k\geq T_\tini$. Given an optimal input sequence at time $k\geq T_\tini$, i.e., $\bu^\ast(k)=\bar\bu(k)+\bu^\ast_g(k)$, define a suboptimal input sequence at time $k+1$ as 
\begin{align}
\label{eq:3:shift}
\bu_s(k+1)&=\bar\bu(k+1)+\bu_g(k+1)\nonumber\\
&=\col(u^\ast(1|k),\ldots,u^\ast(N-1|k),\bar u(N|k))\\&+\col(0,\ldots,0,0),
\end{align}
and let 
\begin{align*}
\by_s(k+1)&=\bar y(k+1)\\&=\col(y^\ast(2|k),\dots,y^\ast(N|k),\bar y(N+1|k))
\end{align*}
 denote the corresponding suboptimal output trajectory. Note that the last output in the suboptimal output sequence satisfies:
\[
	\bar y(N+1|k)=\sum_{i=1}^{T_\tini} a_iy^\ast(N+1-i|k)+\sum_{i=1}^{T_\tini }b_iu^\ast(N+1-i|k).\]
\begin{definition}[\emph{Class $\cK$} functions]
A function $\varphi:\Rset_+\rightarrow\Rset_+$ \emph{belongs to class $\cK$} if it is continuous, strictly increasing and $\varphi(0)=0$. A function $\varphi:\Rset_+\rightarrow\Rset_+$ \emph{belongs to class $\cK_\infty$} if $\varphi\in\cK$ and $\lim_{s\rightarrow\infty}\varphi(s)=\infty$. $\text{id}$ denotes the identity $\cK_\infty$ function, i.e., $\text{id}(s)=s$.
\end{definition}
Next, as proposed in \cite{LazarNMPC2021}, we define a non--minimal state:
\begin{align*}
&\bx_\tini(k):=\\&\col(y(k-T_\tini),\ldots,y(k-1),u(k-T_\tini),\ldots,u(k-1)),
\end{align*} 
and the function  $W(\bx_\tini(k)):=\sum_{i=1}^{T_{\tini}} l(y(k-i),u(k-i))$. In what follows we assume that $r_y=0$ and $r_u=0$ for simplicity of exposition. However, the same proof applies for any constant references that are compatible with an admissible steady--state.
\begin{assumption}[\emph{Terminal stabilizing condition}]
\label{assum:3.1}
For any admissible initial state $\bx_\tini(k)$ there exists a function $\rho\in\cK_\infty$, with $\rho<\text{id}$, a prediction horizon $N\geq T_\tini$ and $\bar u(N|k)\in\Uset$ such that $\bar y(N+1|k)\in\Yset$ and 
\begin{equation}
\label{eq:3:terminal}
l(\bar y(N+1|k),\bar u(N|k))-(\text{id}-\rho)\circ l(y(k-T_\tini),u(k-T_\tini))\leq 0.
\end{equation}
\end{assumption}
\begin{theorem}[\emph{Stability of GDPC}] Suppose that there exist $\alpha_{1,l}, \alpha_{2,l}, \alpha_{2,J}\in\cK_\infty$ such that $\forall (y,u)\in\Yset\times\Uset$
\begin{subequations}
\label{eq:3:bounds}
\begin{align}
\alpha_{1,l}(\|\col(y,u)\|)&\leq l(y,u)\leq\alpha_{2,l}(\|\col(y,u)\|),\label{eq:3:boundsa}\\
J(\by^\ast(k),\bu^\ast(k))&\leq \alpha_{2,J}(\|\bx_\tini(k)\|),\quad\forall \bx_\tini(k)\in\cN,\label{eq:3:boundsb}
\end{align}
\end{subequations}
for some proper set $\cN$ with the origin in its interior. Furthermore, let Assumption~\ref{assum:3.1} hold and suppose that problem \eqref{eq:3.sGDPC} is feasible for all $k\geq T_\tini$. Then system \eqref{eq:2.1} in closed--loop with the GDPC algorithm that solves problem \eqref{eq:3.sGDPC} is asymptotically stable.
\end{theorem}
\begin{proof}
As done in \cite{LazarNMPC2021} for the DeePC algorithm, we consider the following storage function 
\[V(\bx_\tini(k)):=J(\bu^\ast(k),\by^\ast(k))+W(\bx_\tini(k)),\]
and we will prove that it is positive definite and it satisfies a dissipation inequality. First, from \eqref{eq:3:boundsa} and by Lemma~14 in \cite{LazarNMPC2021} we obtain that there exist $\alpha_{1,V},\alpha_{2,V}\in \cK_\infty$ such that
\[\alpha_{1,V}(\|\bx_\tini(k)\|)\leq V(\bx_\tini(k))\leq \alpha_{2,V}(\|\bx_\tini(k)\|).\]
Then, define the supply function
\begin{align*}
&s(y(k-T_\tini),u(k-T_\tini)):=\\
&l(\bar y(N+1|k),\bar u(N|k))-l(y(k-T_\tini),u(k-T_\tini)).
\end{align*}
By the principle of optimality, it holds that
\begin{align*}
V&(\bx_\tini(k+1))-V(\bx_\tini(k))=\\&J(\by^\ast(k+1),\bu^\ast(k+1))+W(\bx_\tini(k+1))
\\&-J(\by^\ast(k),\bu^\ast(k))-W(\bx_\tini(k))\\
&\leq J(\by_s(k+1),\bu_s(k+1))-J(\by^\ast(k),\bu^\ast(k))\\
& + W(\bx_\tini(k+1))-W(\bx_\tini(k))\\
&=l(\bar y(N+1|k),\bar u(N|k))-l(y(0|k),u(0|k))\\
& + l(y(0|k),u(0|k))-l(y(k-T_\tini),u(k-T_\tini))\\
&=s(y(k-T_\tini),u(k-T_\tini)).
\end{align*}
Hence, the storage function $V(\bx_\tini(k))$ satisfies a dissipation inequality along closed--loop trajectories. Since by \eqref{eq:3:terminal} the supply function satisfies
\begin{align*}
s(y(k-T_\tini),u(k-T_\tini))\leq -\rho\circ l(y(k-T_\tini),u(k-T_\tini)),
\end{align*}
the claim then follows from Corollary~17 in \cite{LazarNMPC2021}.
\end{proof}
Notice that condition \eqref{eq:3:terminal} corresponds to a particular case of condition (25) employed in Corollary~17 in \cite{LazarNMPC2021}, i.e., for  $M=1$. 
\begin{remark}[\emph{Terminal stabilizing condition}]
The terminal stabilizing condition \eqref{eq:3:terminal} can be regarded as an implicit condition, i.e., by choosing the prediction horizon $N$ sufficiently large, this condition is more likely to hold. Alternatively, it could be implemented as an explicit constraint in problem \eqref{eq:3.sGDPC} by adding one more input and output at the end of the predicted input and output sequences $\bu(k)$, $\by(k)$, respectively. This also requires including the required additional data in the corresponding $\bar H$ and $H$ Hankel matrices. This yields a convex quadratically constrained QP, which can still be solved efficiently. The stabilizing condition \eqref{eq:3:terminal} can also be used in the regularized DGPC problem \eqref{eq:3.GDPC} to enforce convergence, but in this case a soft constraint implementation is recommended to prevent infeasibility due to noisy data.
\end{remark}
It is worth to point out that the conditions invoked in \cite{Berberich_2020} imply that condition \eqref{eq:3:terminal} holds, i.e., Assumption~\ref{assum:3.1} is less conservative. Indeed, if a terminal equality constraint is imposed in problem \eqref{eq:3.sGDPC}, i.e., $y(N|k)=r_y$, then $\bar u(N|k)=r_u$ is a feasible choice at time $k+1$, which yields $\bar y(N+1|k)=r_y$ and hence, $l(\bar y(N+1|k),\bar u(N|k))=0$. Hence, \eqref{eq:3:terminal} trivially holds since the stage cost $l(y,u)$ is positive definite and $\rho<\text{id}$. Alternatively, one could employ a data--driven method to compute a suitable invariant terminal set in the space of $\bx_\tini$, as proposed in \cite{Berberich_2021}. Tractable data--driven computation of invariant sets is currently possible for ellipsoidal sets, via linear matrix inequalities, which also yields a convex quadratically constrained QP that has to be solved online.   

\section{Simulations results}
\label{sec4}
In this section we consider a benchmark MPC illustrative example from \cite[Section~6.4]{Camacho_2007} based on the flight control of the longitudinal motion of a Boeing 747. After discretization with zero--order--hold for $T_s=0.1[s]$ we obtain a discrete--time linear model as in \eqref{eq:2.1} with:
\begin{equation}
\label{ex:4:dynamics}
\begin{split}
A&=\begin{bmatrix} 0.9997  &  0.0038  & -0.0001 &  -0.0322\\
   -0.0056 &   0.9648  &  0.7446  &  0.0001\\
    0.0020 &  -0.0097  &  0.9543  & -0.0000\\
    0.0001 &  -0.0005 &   0.0978  &  1.0000\end{bmatrix},\\ 
    B&=\begin{bmatrix} 0.0010  &  0.1000\\
   -0.0615 &   0.0183\\
   -0.1133  &  0.0586\\
   -0.0057   & 0.0029\end{bmatrix},\\
   C&=\begin{bmatrix}  1.0000     &    0 &        0   &      0\\
         0  & -1.0000   &      0  &  7.7400\end{bmatrix},
\end{split}
\end{equation}
with two inputs, the throttle $u_1$ and $u_2$, the angle of the elevator, and two outputs, the longitudinal velocity and the climb rate, respectively.
The inputs and outputs are constrained as follows:
\begin{align*}
\Uset&:=\left\{u\in\Rset^2 \ : \ \begin{bmatrix}-20\\-20\end{bmatrix}\leq u \leq \begin{bmatrix}20\\20\end{bmatrix}\right\}\\ 
\Yset&:=\left\{y\in\Rset^2 \ : \ \begin{bmatrix}-25\\-15\end{bmatrix}\leq y \leq \begin{bmatrix}25\\15\end{bmatrix}\right\}.
\end{align*}
The cost function of the predictive controllers is defined as in \eqref{eq:2.DeePC} and \eqref{eq:2:cost} using $\lambda_g = 10^5$, $\lambda_\sigma = 10^7$, $N=20$, $T_{ini}=20$, $Q=10\cdot I_{n_y}$ and $R=0.01\cdot I_{n_u}$. For GDPC, if the suboptimal input sequence is computed as in \eqref{eq:3:ufree}, using shifted optimal sequence from the previous time with $\bar{u}(k+N-1) = u^\ast(k-2+N)$, the regularization cost $l_g(\mathbf{g}(k))$ is defined as in \eqref{eq:2.lg}. If the unconstrained SPC solution is used to compute the suboptimal input sequence, then the regularization cost $l_g(\mathbf{g}(k))=\lambda_g\|\bg(k)\|_2^2$ is used. For DeePC the regularization cost $l_g(\mathbf{g}(k))$ is defined as in \eqref{eq:2.lg}. In this way, all 3 compared predictive control algorithms utilize consistent output predictors.  
The real--time QP (or quadratically constrained QP) control problem is solved using Mosek \cite{mosek} on a laptop with an Intel i7-9750H CPU and 16GB of RAM. 

In what follows, the simulation results are structured into 3 subsections, focusing on nominal, noise--free data performance, noisy data performance for low and high noise variance and comparison with DeePC. The controllers will only start when $T_\tini$ samples have been collected. For the time instants up to $T_\tini$, the system is actuated by a small random input. For the sake of a sound comparison, the random input signal used up to $T_\tini$ is identical for all simulations/predictive controllers. In the data generation experiment, the input sequence is constructed as a PRBS signal between $[-3, 3]$.

\subsection{Noise--free data GDPC performance}
In this simulation we implement the GDPC algorithm with the suboptimal input sequence computed as in \eqref{eq:3:ufree}. Figure \ref{Fig:DeeGPC_noNoise} shows the outputs, inputs and optimized inputs $\bu_g$ over time.

\begin{figure}[h]
	\centering
	\includegraphics[width=1\columnwidth]{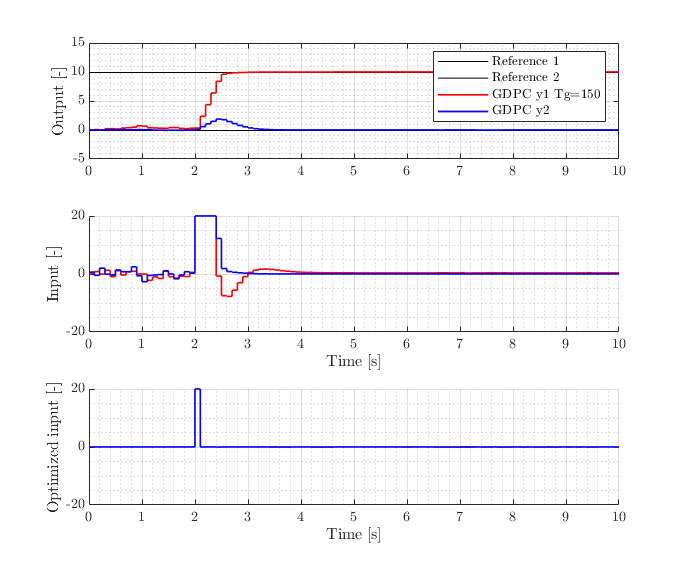}
	\caption{GDPC tracking performance: $T_g=150$, $T=1000$,  $\sigma_w^2I = 0I$.}.
	\label{Fig:DeeGPC_uf}
\end{figure}

We observe that the GDPC closed--loop trajectories converge to the reference values and that the optimized inputs are active only at the start, after which the suboptimal shifted sequence becomes optimal. The stabilizing condition \eqref{eq:3:terminal} is implicitly satisfied along trajectories; when imposed online, the GDPC problem is recursively feasible and yields the same trajectories.

\subsection{Noisy data GDPC performance}
Next we illustrate the performance of GDPC for noisy data with a low and high variance.  
\begin{figure}[h]
	\centering
	\includegraphics[width=1\columnwidth]{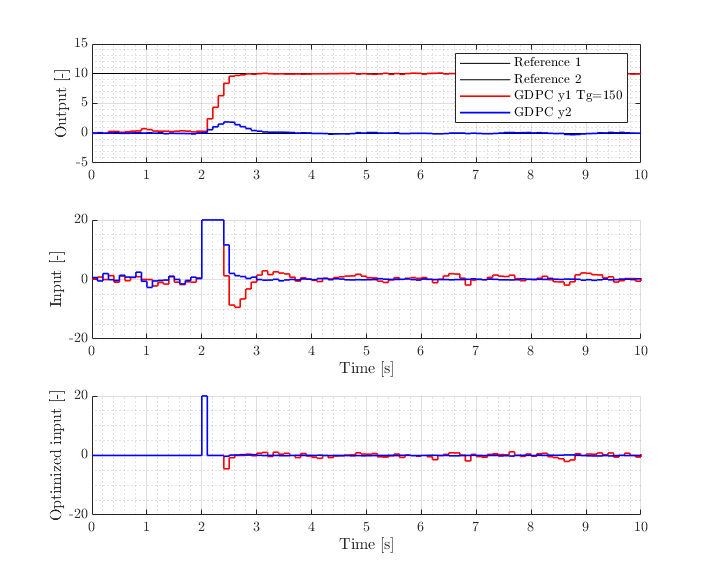}
	\caption{GDPC tracking performance: $T_g=150$, $T=1000$,  $\sigma_w^2I = 0.05I$, $\bar{u}$ as in \eqref{eq:3:ufree}.\label{Fig:DeeGPC_noNoise}}
\end{figure}

Figure~\ref{Fig:DeeGPC_uspc} shows the GDPC response using $\bar{u}(k)$ as defined in \eqref{eq:3:freespc}, i.e., using the unconstrained SPC solution. 
\begin{figure}[h]
	\centering
	\includegraphics[width=1\columnwidth]{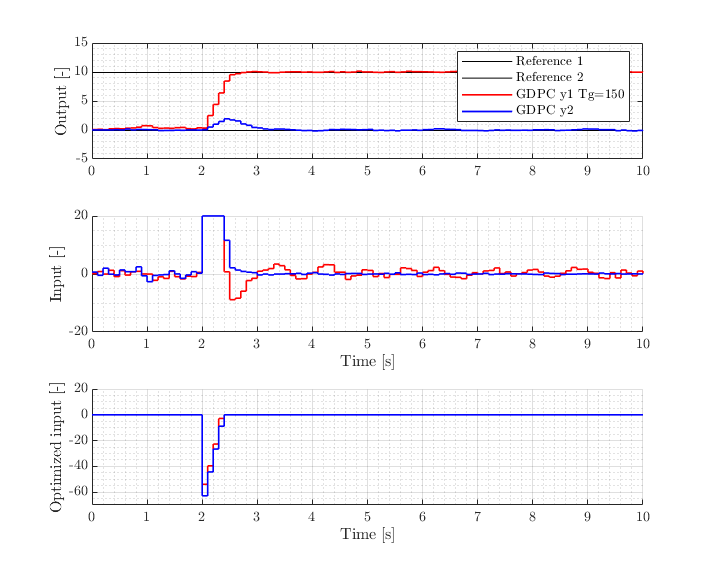}
	\caption{GDPC tracking performance: $T_g=150$, $T=1000$,  $\sigma_w^2I = 0.05I$, $\bar{u}$ as in \eqref{eq:3:freespc}. \label{Fig:DeeGPC_uspc}}
\end{figure}
Although the two different methods to calculate the base line input sequence  $\bar{u}(k)$ show little difference in the resulting total input $u(k)$, the optimized part of the input, $u_g(k)$ shows a notable difference. For the GDPC algorithm that uses the unconstrained SPC we see that the optimized input only acts to enforce constraints, while around steady state the unconstrained SPC becomes optimal.

Next, we show the performance of GDPC for high--variance noise, which also requires a suitable increase of the data size.
\begin{figure}[h]
	\centering
	\includegraphics[width=1\columnwidth]{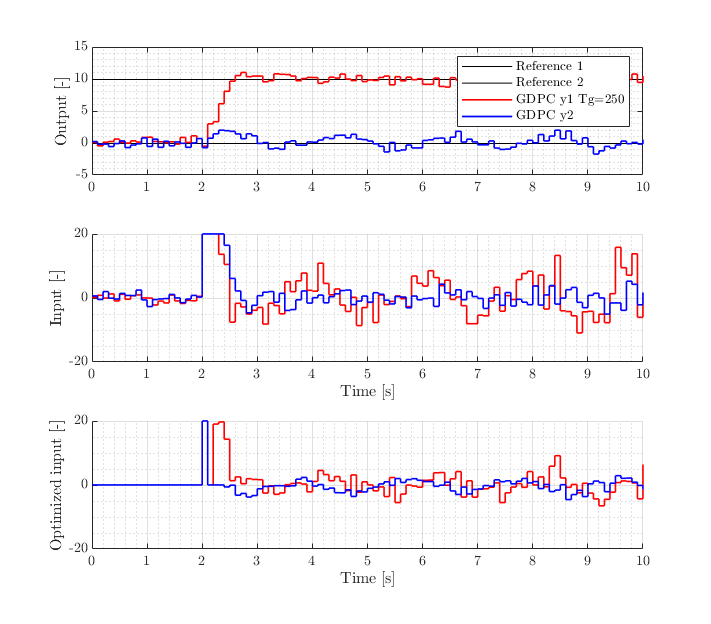}
	\caption{GDPC tracking performance: $T_g=250$, $T=5000$,  $\sigma_w^2I = 0.5I$, $\bar{u}$ as in \eqref{eq:3:ufree}. \label{Fig:DeeGPC_gh_uspc}}
\end{figure}
The GDPC simulation results with $\bar\bu(k)$ computed using the unconstrained SPC solution are shown in Figure~\ref{Fig:DeeGPC_T1000}. We see that in this case the optimized input is active also around steady state, which show that GDPC indeed optimizes the bias variance trade off with respect to the SPC solution.
\begin{figure}[h]
	\centering
	\includegraphics[width=1\columnwidth]{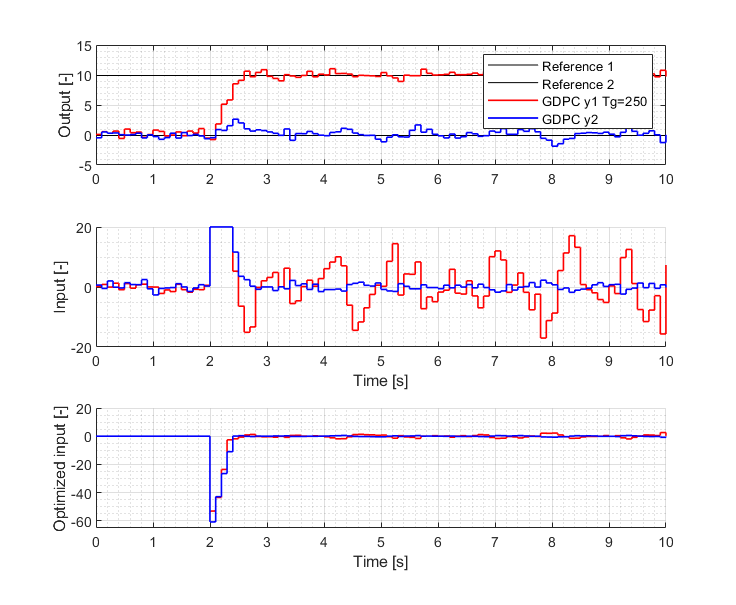}
	\caption{GDPC tracking performance: $T_g=250$, $T=5000$,  $\sigma_w^2I = 0.5I$, $\bar{u}$ as in \eqref{eq:3:freespc}.\label{Fig:DeeGPC_T1000}}
\end{figure}
In Figure \ref{Fig:DeeGPC_gh_uspc} it can be observed that DeePC requires a data sequence of length $750$ to achieve similar performance with GDPC with data lengths  $T_g=250$ (relevant for online complexity), $T=5000$. 
\begin{figure}[h]
	\centering
	\includegraphics[width=1\columnwidth]{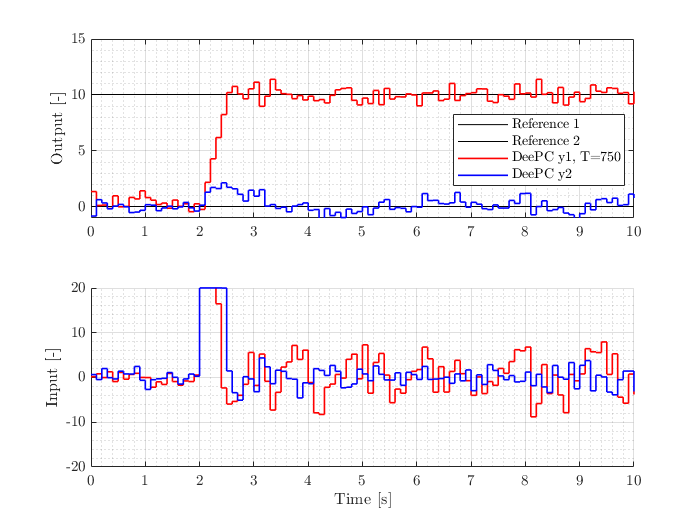}
	\caption{DeePC tracking performance: $T=750$,  $\sigma_w^2I = 0.5I$.\label{Fig:DeeGPC_gh_uspc}}
\end{figure}
The three tested predictive controllers yield  the following average computational time for the high--variance noise simulation: GDPC with $\bar u$ as in \eqref{eq:3:ufree} - 70$\text{ms}$; GDPC with $\bar u$ as in \eqref{eq:3:freespc} - 30$\text{ms}$; DeePC - 680$\text{ms}$. This shows that for high noise and large scale systems the GDPC with $\bar u$ as in \eqref{eq:3:freespc} is the most efficient alternative.

\subsection{Comparison with DeePC over multiple runs}
In this section, we compare the performance of GDPC with DeePC for different sizes of $T$ over multiple runs. The performance can be expressed as \cite{Breschi_2023_Auto}:
\begin{align}
    \mathcal{J} &= \sum_{k=T_\tini}^{t_\text{max}} \|Q^\frac{1}{2}(y(k)-r_y(k))\|_2^2 + \|R^\frac{1}{2}(u(k)-r_u(k))\|_2^2 \\
    \mathcal{J}_u &= \sum_{k=T_\tini}^{t_\text{max}} \|u(k)\|_2^2,
\end{align}
where $t_\text{max}$ is the simulation time and note that the performance scores are computed after the simulation ends, thus using simulated data (not predicted data).  Furthermore, both the data-collecting experiment and the simulation are influenced by noise with variance $\sigma_w^2I = 0.05I$. 
\begin{figure}[h]
\centering
\includegraphics[width=0.8\columnwidth]{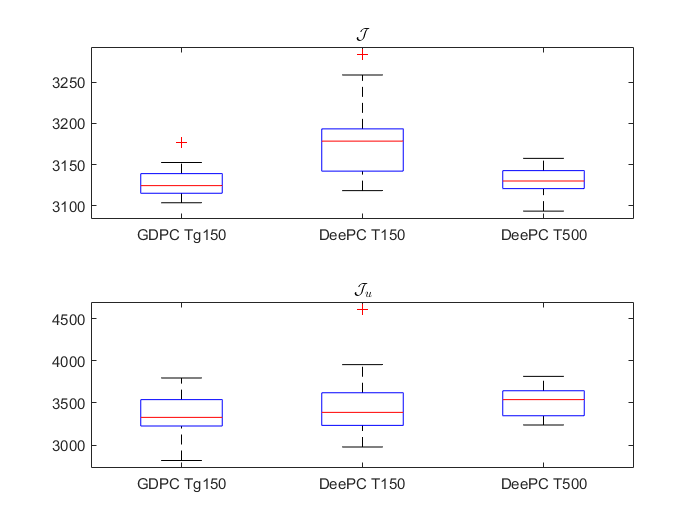}
\caption{GDPC versus DeePC performance, $\sigma_w^2I = 0.05I$.\label{Fig1}}
\end{figure}

As shown in Figure~\ref{Fig1}, we notice that the GDPC predictor with data length $T_g=150$ (Hankel matrix based predictor) and $T=1000$ (least squares based predictor) can match the performance of DeePC with data length $T=500$, while the DeePC performance with $T=150$ is lower. Also, the input cost $\mathcal{J}_u$ is lower for GDPC compared with the the same cost for DeePC with $T=500$.
\begin{table}[h!]
\label{4:table1}
\title{Table 1: Average CPU time for various predictive controllers}
\centering
 \begin{tabular}{|c | c | c | c|} 
 \hline
 & GDPC $T_g=150$ & DeePC $T=150$ & DeePC $T=500$\\
 \hline
$t_\text{cpu}$ & 33ms & 31ms & 235ms \\ 
 \hline
 \end{tabular}
\end{table}

From a computational complexity point of view, as shown in Table~1, the average CPU time of GDPC with $T_g=150$ and $T=1000$ is of the same order as the average CPU time of DeePC with $T=150$, while the average CPU time of DeePC with $T=500$ is about 8 times higher. 

The obtained results validate the fact that GDPC offers more flexibility to optimize the trade off between control performance in the presence of noisy data and online computational complexity. As such, GDPC provides engineers with a practical and robust data--driven predictive controller suitable for real--time implementation.

\section{Conclusions}
\label{sec5}
In this paper we developed a generalized data--driven predictive controller that constructs the predicted input sequence as the sum of a known, suboptimal input sequence and an optimized input sequence that is computed online. This allows us to combine two data--driven predictors: a least squares based, unbiased predictor for computing the suboptimal output trajectory and a Hankel matrix based data--enabled predictor for computing the optimized output trajectory. We have shown that this formulation results in a well--posed data--driven predictive controller with similar stabilizing properties as DeePC. Also, in simulation, we showed that the developed GDPC algorithm can match the performance of DeePC with a large data sequence, for a  smaller data sequence for the Hankel matrix based predictor, which is computationally advantageous.

\end{document}